\documentclass{amsart}%
\usepackage{amsfonts}
\usepackage{amsmath}
\usepackage{amssymb}
\usepackage{graphicx}%
\providecommand{\U}[1]{\protect\rule{.1in}{.1in}}
\newtheorem{theorem}{Theorem}
\theoremstyle{plain}

\newtheorem{proposition}{Proposition}

\numberwithin{equation}{section}
\begin{document}
\title[Subexponential constants in the Bohnenblust--Hille inequality]{A closed formula for subexponential constants in the multilinear
Bohnenblust--Hille inequality}
\author{Diana Marcela Serrano-Rodr\'{\i}guez}
\address{Departamento de Matem\'{a}tica, UFPB, Jo\~{a}o Pessoa, PB, Brazil}
\email{dmserrano0@gmail.com }

\begin{abstract}
For the scalar field $\mathbb{K}=\mathbb{R}$ or $\mathbb{C}$, the multilinear
Bohnenblust--Hille inequality asserts that there exists a sequence of positive
scalars $\left(  C_{\mathbb{K},m}\right)  _{m=1}^{\infty}$ such that%
\[
\left(  \sum\limits_{i_{1},\ldots,i_{m}=1}^{N}\left\vert U(e_{i_{^{1}}}%
,\ldots,e_{i_{m}})\right\vert ^{\frac{2m}{m+1}}\right)  ^{\frac{m+1}{2m}}\leq
C_{\mathbb{K},m}\sup_{z_{1},...,z_{m}\in\mathbb{D}^{N}}\left\vert
U(z_{1},...,z_{m})\right\vert
\]
for all $m$-linear form $U:\mathbb{K}^{N}\times\cdots\times\mathbb{K}%
^{N}\rightarrow\mathbb{K}$ and every positive integer $N$, where $\left(
e_{i}\right)  _{i=1}^{N}$ denotes the canonical basis of $\mathbb{K}^{N}$ and
$\mathbb{D}^{N}$ represents the open unit polydisk in $\mathbb{K}^{N}$. Since
its proof in 1931, the estimates for $C_{\mathbb{K},m}$ have been improved in
various papers. In 2012 it was shown that there exist constants $\left(
C_{\mathbb{K},m}\right)  _{m=1}^{\infty}$ with subexponential growth
satisfying the Bohnenblust-Hille inequality. However, these constants were
obtained via a complicated recursive formula. In this paper, among other
results, we obtain a closed (non-recursive) formula for these constants with
subexponential growth.

\end{abstract}
\maketitle

\section{Introduction}

The complex multilinear Bohnenblust--Hille inequality asserts that for every
positive integer $m\geq1$ there exists a sequence of positive scalars
$C_{\mathbb{K},m}\geq1$ such that
\begin{equation}
\left(  \sum\limits_{i_{1},\ldots,i_{m}=1}^{N}\left\vert U(e_{i_{^{1}}}%
,\ldots,e_{i_{m}})\right\vert ^{\frac{2m}{m+1}}\right)  ^{\frac{m+1}{2m}}\leq
C_{\mathbb{K},m}\sup_{z_{1},...,z_{m}\in\mathbb{D}^{N}}\left\vert
U(z_{1},...,z_{m})\right\vert \label{hypp}%
\end{equation}
for all $m$-linear form $U:\mathbb{K}^{N}\times\cdots\times\mathbb{K}%
^{N}\rightarrow\mathbb{K}$ and every positive integer $N$, where $\left(
e_{i}\right)  _{i=1}^{N}$ is the canonical basis of $\mathbb{K}^{N}$ and
$\mathbb{D}^{N}$ is the open unit polydisk in $\mathbb{K}^{N}$. This
inequality was overlooked for some decades but  it was rediscovered some years
ago and, since then, several works and applications have appeared (see
\cite{annals, defant, defant2, jfa, diniz2, ace, jmaa}). It is well-known
(since the original proof of H.F. Bohnenblust and E. Hille) that the power
$\frac{2m}{m+1}$ is sharp; on the other hand the optimal values of the
constants $C_{\mathbb{K},m}$ are not known. In the case of real scalars the
Bohnenblust--Hille inequality is also valid, but with different constants. In
fact it is known that in the real case%
\[
C_{\mathbb{R},2}=\sqrt{2}%
\]
is optimal (see \cite{diniz2}) and, in the complex case,%
\[
C_{\mathbb{C},2}\leq\frac{2}{\sqrt{\pi}}.
\]
The estimates for these constants are becoming more accurate along the time.
For the complex case we have:

\begin{itemize}
\item $C_{\mathbb{C},m}\leq m^{\frac{m+1}{2m}}2^{\frac{m-1}{2}}$ (1931 -
Bohnenblust and Hille \cite{bh}),

\item $C_{\mathbb{C},m}\leq2^{\frac{m-1}{2}}$ (70's - Kaijser \cite{Ka} and
Davie \cite{d}),

\item $C_{\mathbb{C},m}\leq\left(  \frac{2}{\sqrt{\pi}}\right)  ^{m-1}$ (1995
- Queff\'{e}lec \cite{Q}).
\end{itemize}

\bigskip

Although the optimal constants $C_{\mathbb{K},m}$ are not known, some recent
papers have investigated their asymptotical growth (see \cite{jfa, ace}).Very
recently, quite better estimates, with a surprising subexponential growth,
were obtained in \cite{jfa, jmaa} but the recursive way that these constants
were obtained make the presentation of a closed formula a quite difficult
task. One of the main goals of this paper is to present a closed formula for
the constants with subexponential growth obtained in \cite{jfa, jmaa}.

\section{First remarks}

We begin by recalling the Khinchin inequality:

For any $p>0$, there are constants $A_{p},B_{p}>0$ such that%
\begin{equation}
A_{p}\left(  \sum_{n=1}^{\infty}\left\vert a_{n}\right\vert ^{2}\right)
^{\frac{1}{2}}\leq\left(  \int_{0}^{1}\left\vert \sum_{n=1}^{\infty}a_{n}%
r_{n}\left(  t\right)  \right\vert ^{p}dt\right)  ^{\frac{1}{p}}\leq
B_{p}\left(  \sum_{n=1}^{\infty}\left\vert a_{n}\right\vert ^{2}\right)
^{\frac{1}{2}}. \label{olm}%
\end{equation}
regardless of the $\left(  a_{n}\right)  _{n=1}^{\infty}\in l_{2}.$ Above,
$r_{n}$ represents the $n$-th Rademacher function.

From \cite{haag} we know that the best values of $A_{p}$ are%
\begin{equation}
A_{p}=\left\{
\begin{array}
[c]{c}%
\sqrt{2}\left(  \frac{\Gamma((p+1)/2)}{\sqrt{\pi}}\right)  ^{1/p}\text{, if
}p>p_{0}\\
2^{\frac{1}{2}-\frac{1}{p}}\text{, if }p<p_{0},
\end{array}
\right.  \label{kkklll}%
\end{equation}
where $\Gamma$ denotes the Gamma Function and $1<p_{0}<2$ is so that
\[
\Gamma\left(  \frac{p_{0}+1}{2}\right)  =\frac{\sqrt{\pi}}{2}.
\]
Numerical calculations estimate
\[
p_{0}\approx1.847.
\]

The following result appears in \cite{ace}:

\begin{theorem}
\label{2,1}For all positive integers $n$,%
\begin{align*}
C_{%
%TCIMACRO{\U{211d} }%
%BeginExpansion
\mathbb{R}
%EndExpansion
,2}  &  =2^{\frac{1}{2}},\\
C_{%
%TCIMACRO{\U{211d} }%
%BeginExpansion
\mathbb{R}
%EndExpansion
,3}  &  =2^{\frac{5}{6}}%
\end{align*}
and%
\[
C_{%
%TCIMACRO{\U{211d} }%
%BeginExpansion
\mathbb{R}
%EndExpansion
,n}=2^{\frac{1}{2}}\left(  \frac{C_{%
%TCIMACRO{\U{211d} }%
%BeginExpansion
\mathbb{R}
%EndExpansion
,n-2}}{A_{\frac{2n-4}{n-1}}^{2}}\right)  ^{\frac{n-2}{n}}\text{ for }n>3.
\]
In particular, if $2\leq n\leq14$%
\begin{align*}
C_{%
%TCIMACRO{\U{211d} }%
%BeginExpansion
\mathbb{R}
%EndExpansion
,n}  &  =2^{\frac{n^{2}+6n-8}{8n}}\text{, if }n\text{ is even}\\
C_{%
%TCIMACRO{\U{211d} }%
%BeginExpansion
\mathbb{R}
%EndExpansion
,n}  &  =2^{\frac{n^{2}+6n-7}{8n}}\text{, if }n\text{ is odd.}%
\end{align*}

\end{theorem}

The above theorem allows to obtain a closed formula for the constants. It is
shown in \cite{ace} that for an even positive integer $n>14$,
\[
C_{%
%TCIMACRO{\U{211d} }%
%BeginExpansion
\mathbb{R}
%EndExpansion
,n}=2^{\frac{n+2}{8}}r_{n},
\]
for a certain $r_{n}$ for which numerical computations show that it tends to a
number close to $1.44.$ The formula for $r_{n}$ from \cite{ace} contains a
slight imprecision which affects some decimals of the first constants. Below
we show a correct formula for $r_{n}.$

\begin{proposition}
\label{2,2}If $n>14$ is even$,$ then%
\[
C_{%
%TCIMACRO{\U{211d} }%
%BeginExpansion
\mathbb{R}
%EndExpansion
,n}=2^{\frac{n+2}{8}}r_{n},
\]
with
\begin{equation}
r_{n}=\frac{\pi^{\frac{\left(  n+14\right)  \left(  n-14\right)  }{8n}}%
}{2^{\frac{\left(  n+12\right)  \left(  n-14\right)  -24}{4n}}.\left[
%TCIMACRO{\dprod \limits_{k=7}^{\frac{n-2}{2}}}%
%BeginExpansion
{\displaystyle\prod\limits_{k=7}^{\frac{n-2}{2}}}
%EndExpansion
\left(  \Gamma\left(  \frac{6k+1}{4k+2}\right)  \right)  ^{2k+1}\right]
^{\frac{1}{n}}}. \label{wwqa}%
\end{equation}

\end{proposition}

\begin{proof}
Using the estimates from Theorem\textit{ \ref{2,1} we have}%
\begin{align*}
C_{%
%TCIMACRO{\U{211d} }%
%BeginExpansion
\mathbb{R}
%EndExpansion
,4}  &  =2^{\frac{1}{2}}\left(  \frac{C_{%
%TCIMACRO{\U{211d} }%
%BeginExpansion
\mathbb{R}
%EndExpansion
,2}}{A_{\frac{4}{3}}^{2}}\right)  ^{\frac{2}{4}}\\
C_{%
%TCIMACRO{\U{211d} }%
%BeginExpansion
\mathbb{R}
%EndExpansion
,6}  &  =2^{\frac{1}{2}}\left(  \frac{2^{\frac{1}{2}}\left(  \frac{C_{%
%TCIMACRO{\U{211d} }%
%BeginExpansion
\mathbb{R}
%EndExpansion
,2}}{A_{\frac{4}{3}}^{2}}\right)  ^{\frac{2}{4}}}{A_{\frac{8}{5}}^{2}}\right)
^{\frac{4}{6}}=\frac{\left(  2^{\frac{1}{2}+\frac{1}{2}.\frac{4}{6}}\right)
\left(  C_{%
%TCIMACRO{\U{211d} }%
%BeginExpansion
\mathbb{R}
%EndExpansion
,2}\right)  ^{\frac{2}{4}.\frac{4}{6}}}{\left(  A_{\frac{4}{3}}^{2}\right)
^{\frac{2}{4}.\frac{4}{6}}\left(  A_{\frac{8}{5}}^{2}\right)  ^{\frac{4}{6}}%
}\\
C_{%
%TCIMACRO{\U{211d} }%
%BeginExpansion
\mathbb{R}
%EndExpansion
,8}  &  =\frac{2^{\frac{1}{2}+\left(  \frac{1}{2}+\frac{1}{2}.\frac{4}%
{6}\right)  \frac{6}{8}}\left(  C_{%
%TCIMACRO{\U{211d} }%
%BeginExpansion
\mathbb{R}
%EndExpansion
,2}\right)  ^{\frac{2}{4}.\frac{4}{6}.\frac{6}{8}}}{\left(  A_{\frac{4}{3}%
}^{2}\right)  ^{\frac{2}{4}.\frac{4}{6}.\frac{6}{8}}\left(  A_{\frac{8}{5}%
}^{2}\right)  ^{\frac{4}{6}.\frac{6}{8}}\left(  A_{\frac{12}{7}}^{2}\right)
^{\frac{6}{8}}}%
\end{align*}

and so on. Hence%
\begin{equation}
C_{%
%TCIMACRO{\U{211d} }%
%BeginExpansion
\mathbb{R}
%EndExpansion
,n}=\frac{d_{n}}{s_{n}} \label{555}%
\end{equation}
with
\[
s_{n}=\left(  A_{\frac{4}{3}}^{2}\right)  ^{\frac{2}{4}.\frac{4}{6}%
...\frac{n-2}{n}}\left(  A_{\frac{8}{5}}^{2}\right)  ^{\frac{4}{6}.\frac{6}%
{8}...\frac{n-2}{n}}\left(  A_{\frac{12}{7}}^{2}\right)  ^{\frac{6}{8}%
.\frac{8}{10}...\frac{n-2}{n}}...\left(  A_{\frac{2n-4}{n-1}}^{2}\right)
^{\frac{n-2}{n}}%
\]

and
\[
d_{n}=2^{\frac{1}{2}+\frac{1}{2}\left(  \frac{n-2}{n}\right)  +\frac{1}%
{2}\left(  \frac{n-4}{n}\right)  +\frac{1}{2}\left(  \frac{n-6}{n}\right)
+...+\frac{1}{2}\left(  \frac{n-\left(  n-4\right)  }{n}\right)  }\sqrt
{2}^{\frac{2}{n}}.
\]

For $p=\frac{2n-4}{n-1}$ and $2\leq n\leq14$,$~$we have $p<1.847$. So
\[
A_{p}=2^{\frac{1}{2}-\frac{1}{p}}%
\]

and, for $n>14$, we have $p>p_{0}$ and%
\[
A_{p}=2^{\frac{1}{2}}\left(  \frac{\Gamma\left(  \left(  p+1\right)
/2\right)  }{\sqrt{\pi}}\right)  ^{\frac{1}{p}}.
\]
We thus have
\begin{align*}
s_{n}  &  =\left(  A_{\frac{4}{3}}^{2}\right)  ^{\frac{2}{4}.\frac{4}%
{6}...\frac{n-2}{n}}\left(  A_{\frac{8}{5}}^{2}\right)  ^{\frac{4}{6}.\frac
{6}{8}...\frac{n-2}{n}}...\left(  A_{\frac{24}{13}}^{2}\right)  ^{\frac
{12}{14}.\frac{14}{16}...\frac{n-2}{n}}\left(  A_{\frac{28}{15}}^{2}\right)
^{\frac{14}{16}.\frac{16}{18}...\frac{n-2}{n}}...\left(  A_{\frac{2n-4}{n-1}%
}^{2}\right)  ^{\frac{n-2}{n}}\\
&  =\left(  2^{-\frac{1}{4}}\right)  ^{\frac{2}{\left(  \frac{n}{2}\right)  }%
}\left(  2^{-\frac{1}{8}}\right)  ^{\frac{4}{\left(  \frac{n}{2}\right)  }%
}...\left(  2^{-\frac{1}{24}}\right)  ^{\frac{12}{\left(  \frac{n}{2}\right)
}}\times\\
&  \times\left(  \sqrt{2}\left(  \frac{\Gamma\left(  \frac{43}{30}\right)
}{\sqrt{\pi}}\right)  ^{\frac{15}{28}}\right)  ^{\frac{14}{\frac{n}{2}}%
}...\left(  \sqrt{2}\left(  \frac{\Gamma\left(  \left(  \frac{2n-4}%
{n-1}+1\right)  /2\right)  }{\sqrt{\pi}}\right)  ^{\frac{1}{p}}\right)
^{\frac{n-2}{n}}\\
&  =\left(  2^{-\frac{6}{n}}\right)  \left(  \sqrt{2}\left(  \frac
{\Gamma\left(  \frac{43}{30}\right)  }{\sqrt{\pi}}\right)  ^{\frac{15}{28}%
}\right)  ^{\frac{14}{\frac{n}{2}}}...\left(  \sqrt{2}\left(  \frac
{\Gamma\left(  \left(  \frac{2n-4}{n-1}+1\right)  /2\right)  }{\sqrt{\pi}%
}\right)  ^{\frac{1}{p}}\right)  ^{\frac{n-2}{n}}\\
&  =2^{\frac{\left(  n+12\right)  \left(  n-14\right)  -24}{4n}}\left(
%TCIMACRO{\dprod \limits_{k=7}^{\frac{n-2}{2}}}%
%BeginExpansion
{\displaystyle\prod\limits_{k=7}^{\frac{n-2}{2}}}
%EndExpansion
\left(  \frac{\Gamma\left(  \frac{6k+1}{4k+2}\right)  }{\sqrt{\pi}}\right)
^{2k+1}\right)  ^{\frac{1}{n}}\\
&  =2^{\frac{\left(  n+12\right)  \left(  n-14\right)  -24}{4n}}\left(
%TCIMACRO{\dprod \limits_{k=7}^{\frac{n-2}{2}}}%
%BeginExpansion
{\displaystyle\prod\limits_{k=7}^{\frac{n-2}{2}}}
%EndExpansion
\left(  \Gamma\left(  \frac{6k+1}{4k+2}\right)  \right)  ^{2k+1}\right)
^{\frac{1}{n}}\left(  \pi^{\frac{\left(  n+14\right)  \left(  n-14\right)
}{8n}}\right)  ^{-1}\\
&  =\frac{1}{r_{n}}.
\end{align*}

On the other hand a simple calculation shows that%
\[
d_{n}=2^{\frac{n+2}{8}}%
\]
and from (\ref{555}) we obtain%
\[
C_{%
%TCIMACRO{\U{211d} }%
%BeginExpansion
\mathbb{R}
%EndExpansion
,n}=2^{\frac{n+2}{8}}r_{n}.
\]
\bigskip
\end{proof}

Below we compare the values of the $r_{n}$ from (\ref{wwqa}) and the $r_{n}$
from \cite{ace}:%

\[%
\begin{tabular}
[c]{|l|l|l|}\hline
$n$ & $r_{n}$ (\ref{wwqa}) & $r_{n}$ (\cite{ace})\\\hline
$30$ & $1.387$ & $1.375$\\\hline
$50$ & $1.\,404$ & $1.397$\\\hline
$100$ & $1.\,420$ & $1.416$\\\hline
$250$ & $1.\,431$ & $1.429$\\\hline
$500$ & $1.435$ & $1.434$\\\hline
$1,000$ & $1.\,4374$ & $1.4371$\\\hline
$10,000$ & $1.43989$ & $1.43986$\\\hline
$100,000$ & $1.44021$ & $1.44021$\\\hline
\end{tabular}
\ \ \ \ \
\]

Hence, although the formulas for $r_{n}$ are different its values are very
close and, as in \cite{ace}, numerical estimates indicate that%
\[
\lim_{n\rightarrow\infty}r_{n}\approx1.44025.
\]
We conjecture that%
\[
\lim_{n\rightarrow\infty}r_{n}=\frac{e^{1-\frac{1}{2}\gamma}}{\sqrt{2}},
\]
where $\gamma$ denotes the Euler constant.

\section{Main results}

In \cite{jfa} it was shown that there is a constant $D$ (probably very close
to $1.44$) so that the sequence $\left(  C_{n}\right)  _{n=1}^{\infty}$ given
by%
\begin{align*}
C_{2n}  &  =DC_{n}\\
C_{2n+1}  &  =D\left(  C_{n}\right)  ^{\frac{2n}{4n+2}}\left(  C_{n+1}\right)
^{\frac{2n+2}{4n+2}},
\end{align*}
with%
\[
C_{1}=1\text{ and }C_{2}=\sqrt{2}%
\]
in the real case and
\[
C_{1}=1\text{ and }C_{2}=\frac{2}{\sqrt{\pi}}%
\]
in the complex case, satisfies the Bohnenblust--Hille inequality and,
moreover, this sequence is subexponential. From now on $C_{n}$ will denote the
numbers given by the above formulas.

In this section we present a closed formula for these constants. Given a
positive integer $n$, it is plain that it can be written (in an unique way)
as
\begin{equation}
n=2^{k}-l,\text{ } \label{et}%
\end{equation}
where $k$ is the smaller positive integer such that $2^{k}\geq n$ and $0\leq
l<2^{k-1}$.

\begin{theorem}
If $n\geq3$ is written as (\ref{et}), then%
\begin{equation}
C_{n}=D^{k-1}C_{2}^{\frac{n-l}{n}}\text{, if }l\leq2^{k-2} \label{*}%
\end{equation}
and%
\begin{equation}
C_{n}=D^{\frac{n\left(  k-1\right)  +2^{k-1}-2l}{n}}C_{2}^{\frac{2^{k-1}}{n}%
}\text{, if }2^{k-2}<l<2^{k-1} \label{**}%
\end{equation}
where
\begin{align*}
C_{2}  &  =\sqrt{2}\text{, for real scalars}\\
C_{2}  &  =\frac{2}{\sqrt{\pi}}\text{ for complex scalars}%
\end{align*}

\end{theorem}

\begin{proof}
Since $n\geq3$, note that $k\geq2$.

We proceed by induction. Suppose the result valid for all $m\leq$ $n.$

Let
\[
n+1=2^{k}-l
\]
with $l$ and $k$ so that $k$ is the smaller positive integer such that
$2^{k}\geq n+1$ and $0\leq l<2^{k-1}$.

\begin{itemize}
\item \textbf{First Case:} $l$ is even.
\end{itemize}

In this case $n+1$ is even and
\begin{equation}
C_{n+1}=D\left(  C_{\frac{n+1}{2}}\right)  \label{par}%
\end{equation}
with
\[
\frac{n+1}{2}=2^{k-1}-\frac{l}{2}.
\]
\ By induction hypothesis, the result is valid for $C_{\frac{n+1}{2}}.$ We
have two possible subcases for $\frac{l}{2}:$

\textbf{Subcase 1a -}%

\begin{equation}
\frac{l}{2}\leq2^{\left(  k-1\right)  -2}=2^{k-3}. \label{zzz}%
\end{equation}

\textbf{Subcase 1b -}%
\[
2^{\left(  k-1\right)  -2}<\frac{l}{2}<2^{\left(  k-1\right)  -1},
\]
i.e.,%

\begin{equation}
2^{k-3}<\frac{l}{2}<2^{k-2}. \label{zzz2}%
\end{equation}

If (\ref{zzz}) occurs, note that $l\leq2^{k-2},$ from (\ref{par}) we have%
\begin{align*}
C_{n+1}  &  =D\left(  D^{k-2}C_{2}^{\frac{n+1-l}{n+1}}\right) \\
&  =D^{k-1}C_{2}^{\frac{\left(  n+1\right)  -l}{n+1}},
\end{align*}
and this is what we need.

If (\ref{zzz2}) occurs, note that $2^{k-2}<l<2^{k-1}.$ From (\ref{par}) we
have%
\begin{align*}
C_{n+1}  &  =D\left(  C_{\frac{n+1}{2}}\right) \\
&  =D\left(  D^{\frac{\left(  n+1\right)  \left(  k-2\right)  +2^{k-1}%
-2l}{n+1}}C_{2}^{\frac{2^{k-1}}{n+1}}\right) \\
&  =D^{\frac{^{\left(  n+1\right)  \left(  k-1\right)  +2^{k-1}-2l}}{n+1}%
}C_{2}^{\frac{2^{k-1}}{n+1}}%
\end{align*}
and again we get the desired result.

\begin{itemize}
\item \textbf{Second Case:} $l$ is odd.

In this case $n+1$ is odd and
\begin{align}
C_{n+1}  &  =D\left(  C_{\frac{\left(  n+1\right)  -1}{2}}\right)
^{\frac{\frac{\left(  n+1\right)  -1}{2}}{n+1}}\left(  C_{\frac{\left(
n+1\right)  +1}{2}}\right)  ^{\frac{\frac{\left(  n+1\right)  +1}{2}}{n+1}%
}\label{impar}\\
&  =D\left(  C_{\frac{n}{2}}\right)  ^{\frac{\frac{n}{2}}{\left(  n+1\right)
}}\left(  C_{\frac{n+2}{2}}\right)  ^{\frac{\frac{n+2}{2}}{\left(  n+1\right)
}}\nonumber
\end{align}
Since $n+1=2^{k}-l\,$, we have
\[
n=2^{k}-\left(  l+1\right)  ,
\]
and
\[
n+2=2^{k}-\left(  l-1\right)  .
\]
Since $0\leq l<2^{k-1}$, and $l\,\ $is odd, then
\[
0\leq l+1<2^{k-1}\text{, or }l+1=2^{k-1}%
\]
and we have two subcases:
\end{itemize}

\textbf{Subcase 2a -}%

\begin{equation}
n=2^{k-1}-0,\text{ and }~n+2=2^{k}-\left(  l-1\right)  , \label{zzz3}%
\end{equation}
with%
\[
l=2^{k-1}-1
\]

\textbf{Subcase 2b -}
\begin{equation}
~n=2^{k}-\left(  l+1\right)  ,\text{ and }~n+2=2^{k}-\left(  l-1\right)  ,
\label{zzz4}%
\end{equation}
with%
\[
l<2^{k-1}-1.
\]

If (\ref{zzz3}) holds, then $l=2^{k-1}-1.$ Since%
\[
\frac{n}{2}=2^{k-2},
\]
then $C_{\frac{n}{2}}$ is of the form (\ref{*}) and, since%
\[
\frac{n+2}{2}=2^{k-1}-\frac{\left(  l-1\right)  }{2}\text{ and }2^{k-3}%
<\frac{l-1}{2}<2^{k-2},
\]
then $C_{\frac{n+2}{2}}$ is of the form (\ref{**}). We this have%
\[
C_{\frac{n}{2}}=D^{\left(  k-2\right)  -1}C_{2}%
\]
and%
\[
C_{\frac{n+2}{2}}=D^{\frac{\frac{n+2}{2}\left(  k-2\right)  +2^{\left(
k-1\right)  -1}-2\left(  \frac{l-1}{2}\right)  }{\frac{n+2}{2}}}C_{2}%
^{\frac{2^{\left(  k-1\right)  -1}}{\frac{n+2}{2}}}.
\]

From (\ref{impar}), we get%
\begin{align*}
C_{n+1}  &  =D\left(  C_{\frac{n}{2}}\right)  ^{\frac{\frac{n}{2}}{\left(
n+1\right)  }}\left(  C_{\frac{n+2}{2}}\right)  ^{\frac{\frac{n+2}{2}}{\left(
n+1\right)  }}\\
&  =D\left(  D^{k-3}C_{2}\right)  ^{\frac{\frac{n}{2}}{\left(  n+1\right)  }%
}\left(  D^{\frac{\frac{n+2}{2}\left(  k-2\right)  ~+\frac{2^{k-1}}%
{2}~-2\left(  \frac{l-1}{2}\right)  }{\frac{n+2}{2}}}C_{2}^{\frac{2^{k-2}%
}{\frac{n+2}{2}}}\right)  ^{\frac{\frac{n+2}{2}}{\left(  n+1\right)  }}\\
&  =D^{\frac{\left(  n+1\right)  \left(  k-1\right)  +2^{k-1}-2l}{n+1}}%
C_{2}^{\frac{2^{k-1}}{n+1}}%
\end{align*}
and we have the desired result.

In the case that (\ref{zzz4}) holds, we have
\begin{align*}
l+1  &  <2^{k-1},\\
\frac{n}{2}  &  =2^{k-1}-\frac{\left(  l+1\right)  }{2}%
\end{align*}
and
\[
\frac{n+2}{2}=2^{k-1}-\frac{\left(  l-1\right)  }{2}.
\]
We have three sub-subcases:

\textbf{Sub-subcase 2ba -}%

\begin{equation}
2^{k-2}<l+1<2^{k-1},\text{ and }~2^{k-2}<l-1<2^{k-1} \label{zzz5}%
\end{equation}

\textbf{Sub-subcase 2bb -}$~$%
\begin{equation}
2^{k-2}<l+1<2^{k-1},\text{ and }l-1=2^{k-2} \label{zzz6}%
\end{equation}

\textbf{Sub-subcase 2bc -}%
\begin{equation}
~l-1<l+1\leq2^{k-2}. \label{zzz7}%
\end{equation}

If (\ref{zzz5}) holds, note that
\[
2^{k-2}<l-1<l<l+1<2^{k-1},
\]
and this $C_{n+1}$ is of the form (\ref{**}). Therefore
\[
2^{k-3}<\frac{l+1}{2}<2^{k-2}\text{, and~}2^{k-3}<\frac{l-1}{2}<2^{k-2}%
\]
and this $C_{\frac{n}{2}}$ and $C_{\frac{n+2}{2}}\,$are written in the form
(\ref{**}); now, from (\ref{impar}) we have%
\begin{align*}
C_{n+1} &  =D\left(  C_{\frac{n}{2}}\right)  ^{\frac{\frac{n}{2}}{\left(
n+1\right)  }}\left(  C_{\frac{n+2}{2}}\right)  ^{\frac{\frac{n+2}{2}}{\left(
n+1\right)  }}\\
&  =D\left(  D^{\frac{\frac{n}{2}\left(  k-2\right)  +2^{k-2}-2\left(
\frac{l+1}{2}\right)  }{\frac{n}{2}}}C_{2}^{\frac{2^{k-2}}{\frac{n}{2}}%
}\right)  ^{\frac{\frac{n}{2}}{n+1}}\left(  D^{\frac{\frac{n+2}{2}\left(
k-2\right)  +2^{k-2}-2\left(  \frac{l-1}{2}\right)  }{\frac{n+2}{2}}}%
C_{2}^{\frac{2^{k-2}}{\frac{n+2}{2}}}\right)  ^{\frac{\frac{n+2}{2}}{n+1}}\\
&  =D^{\frac{\left(  n+1\right)  \left(  k+1\right)  +2^{k-1}-2l}{n+1}}%
C_{2}^{\frac{2^{k-1}}{n+1}}.
\end{align*}
If (\ref{zzz6}) holds, note that
\[
2^{k-2}=l-1<l<l+1<2^{k-1},
\]
and we need to obtain a formula like (\ref{**}). Since
\[
2^{k-3}<\frac{l+1}{2}<2^{k-2}\text{, and }\frac{l-1}{2}=2^{k-3}%
\]
then $C_{\frac{n}{2}}$ is represented by (\ref{**}) and $C_{\frac{n+2}{2}%
}\,is$ of the form (\ref{*}). So, from (\ref{impar}), we have%
\begin{align*}
C_{n+1} &  =D\left(  D^{\frac{\frac{n}{2}\left(  k-2\right)  +2^{k-2}-2\left(
\frac{l+1}{2}\right)  }{\frac{n}{2}}}C_{2}^{\frac{2^{k-2}}{\frac{n}{2}}%
}\right)  ^{\frac{\frac{n}{2}}{n+1}}\left(  D^{k-2}C_{2}^{\frac{\frac{n+2}%
{2}-2^{k-3}}{\frac{n+2}{2}}}\right)  ^{\frac{\frac{n+2}{2}}{n+1}}\\
&  =D.D^{\frac{\frac{n}{2}\left(  k-2\right)  +2^{k-2}-\left(  l+1\right)
}{n+1}}.D^{\frac{\left(  k-2\right)  \left(  \frac{n+2}{2}\right)  }{n+1}%
}.C_{2}^{\frac{2^{k-2}}{n+1}}.C_{2}^{\frac{\frac{n+2}{2}-2^{k-3}}{n+1}}\\
&  =D^{\frac{\left(  n+1\right)  +\left(  k-2\right)  \left(  n+1\right)
+2^{k-2}-2^{k-2}-2}{n+1}}.C_{2}^{\frac{2^{k-3}+\frac{n+2}{2}}{n+1}}%
\end{align*}
Since
\[
2^{k-1}-2l=2^{k-1}-2\left(  2^{k-2}+1\right)  =-2,
\]
and
\[
\frac{n+2}{2}=2^{k-1}-\left(  \frac{l-1}{2}\right)  =2^{k-1}-2^{k-3},
\]
then%
\[
C_{n+1}=D^{\frac{\left(  n+1\right)  \left(  k-1\right)  +2^{k-1}-2l}{n+1}%
}C_{2}^{\frac{2^{k-1}}{n+1}}.
\]
Finally, if we have (\ref{zzz7}), note that
\[
l-1<l<l+1\leq2^{k-2},
\]
and then $C_{n+1}$ must be of the form (\ref{*}). Hence
\[
\frac{l-1}{2}<\frac{l+1}{2}\leq2^{k-3},
\]
and $C_{\frac{n}{2}}$ and $C_{\frac{n+2}{2}}$ are written in the form of
(\ref{*}). Thus, again using (\ref{impar}) we have
\begin{align*}
C_{n+1} &  =D\left(  D^{k-2}C_{2}^{\frac{\frac{n}{2}-\frac{l+1}{2}}{\frac
{n}{2}}}\right)  ^{\frac{\frac{n}{2}}{n+1}}\left(  D^{k-2}C_{2}^{\frac
{\frac{n+2}{2}-\frac{l-1}{2}}{\frac{n+2}{2}}}\right)  ^{\frac{\frac{n+2}{2}%
}{n+1}}\\
&  =D.D^{\frac{\left(  k-2\right)  \frac{n}{2}}{n+1}}.D^{\frac{\left(
k-2\right)  \frac{n+2}{2}}{n+1}}.C_{2}^{\frac{\frac{n}{2}-\frac{l+1}{2}}{n+1}%
}.C_{2}^{\frac{\frac{n+2}{2}-\frac{l-1}{2}}{n+1}}\\
&  =D^{\frac{\left(  n+1\right)  +\left(  k-2\right)  \left(  \frac{n}%
{2}+\frac{n+2}{2}\right)  }{n+1}}C_{2}^{\frac{2^{k-1}-\left(  \frac{l+1}%
{2}\right)  -\left(  \frac{l+1}{2}\right)  }{n+1}}C_{2}^{\frac{2^{k-1}-\left(
\frac{l-1}{2}\right)  -\left(  \frac{l-1}{2}\right)  }{n+1}}\\
&  =D^{k-1}C_{2}^{\frac{\left(  n+1\right)  -l}{n+1}},
\end{align*}
and the proof is done.
\end{proof}

\end{document}